\documentclass[a4paper,onecolumn,superscriptaddress,10pt]{compositionalityarticle}
\pdfoutput=1

\usepackage[utf8]{inputenc}
\usepackage[english]{babel}
\usepackage[T1]{fontenc}
\usepackage{amsmath}
\usepackage{hyperref}
\usepackage{tikz}
\usepackage{lipsum}
\usepackage[numbers,sort&compress]{natbib}
   
\newcommand{\authorsforheader}{Mary}
\newcommand{\compositionalityIssue}{4}
\newcommand{\compositionalityVolume}{8}
\newcommand{\paperdoinumber}{10.46298/compositionality-\compositionalityVolume-\compositionalityIssue}
\newcommand{\paperdoi}{https://doi.org/\paperdoinumber}
\newcommand{\receiveddate}{2025-10-31}
\newcommand{\accepteddate}{2026-08-31}

\usepackage{amsfonts}
\usepackage{amssymb}

\usepackage{bm}
\usepackage{graphicx}
\usepackage{mathrsfs}
\usepackage{fdsymbol}
\usepackage{multirow}
\usepackage{colortbl}
\usepackage{enumerate}
\usepackage{tikz}
\usepackage{amsthm}
\usepackage[all]{xy}
\xyoption{2cell}
\UseAllTwocells
\newcommand{\xym}{\ensuremath \xymatrix@1}

\definecolor{darkblue}{rgb}{0.,0.,0.4}
\definecolor{darkred}{rgb}{0.5,0.,0.}

\definecolor{lightgray}{gray}{0.85}
\definecolor{dgray}{gray}{0.35}
\definecolor{sgray}{gray}{0.55}

\numberwithin{equation}{section}
\numberwithin{figure}{section}

\def\.{\cdot}
\def\:{\odot}

    \def\f{{\mathrm{f}}}

    \def\r{{\mathrm{r}}}

\def\DD{{\mathbb{D}}} \def\EE{{\mathbb{E}}}

\def\MM{{\mathbb{M}}} \def\NN{{\mathbb{N}}}

 \def\bB{{\mathbf{B}}}

  \def\cC{{\mathcal{C}}}
\def\cD{{\mathcal{D}}} \def\cE{{\mathcal{E}}} 
  
  \def\cL{{\mathcal{L}}}
\def\cM{{\mathcal{M}}}  
\def\cP{{\mathcal{P}}}  \def\cR{{\mathcal{R}}}

\def\leqL{{\;\leq_{\cL}\;}} \def\leqR{{\;\leq_{\cR}\;}} 

\def\bmx{\left[\begin{array}}
\def\emx{\end{array}\right]}

\def\beqn{\begin{eqnarray*}}
\def\eeqn{\end{eqnarray*}}
\def\barr{\begin{array}}
\def\earr{\end{array}}

\def\benn{\begin{enumerate}[$(1)$]}
\def\ben{\begin{enumerate}}
\def\een{\end{enumerate}}

\def\bit{\begin{itemize}\setlength\itemsep{-0.3em}}
\def\eit{\end{itemize}}

\def\cM{\mathcal{M}}

\def\beqn{\begin{eqnarray*}}
\def\eeqn{\end{eqnarray*}}
\def\barr{\begin{array}}
\def\earr{\end{array}}

\def\bit{\begin{itemize}}
\def\eit{\end{itemize}}

\def\een{\end{enumerate}}
\def\l({\left(}
\def\r){\right)}

\def\bmx{\left[\begin{array}}
\def\emx{\end{array}\right]}

\newtheorem{theorem}{Theorem}[section]

\newtheorem{lemma}[theorem]{Lemma}
\newtheorem{proposition}[theorem]{Proposition}
\newtheorem{corollary}[theorem]{Corollary}
\newtheorem{definition}[theorem]{Definition}
\newtheorem{example}[theorem]{\textit{Example}}

\usepackage{hyperref,doi}
\usepackage{tikz}

\newcommand{\lCE}{%
  \raisebox{0.1cm}{\,\begin{tikzpicture}\draw[->>] (0,0) -- (.4,0);\end{tikzpicture}\,}}
\newcommand{\lCM}{%
  \raisebox{0.1cm}{\,\begin{tikzpicture}\draw[>->] (0,0) -- (.4,0);\end{tikzpicture}\,}}

\title[A Categorical Realization of the Category of Monoids]{A Categorical Realization of the (2-)Category of Monoids via Sch{\"u}tzenberger Categories and Strict Factorization Systems}

\author{Xavier Mary}
\email{xavier.mary@parisnanterre.fr}
\homepage{https://xmary.perso.math.cnrs.fr/}
\orcid{0000-0003-3655-5799}
\affiliation{Universit\'e Paris Nanterre, France, Laboratoire Modal'X}

\date{}

\begin{document}

\maketitle

\begin{abstract}
We construct a category equivalent to the category $\mathbf{Mon}$ of monoids and monoid homomorphisms, based on the Sch{\"u}tzenberger category of semigroups and categories with strict factorization systems. This equivalence is then extended to the category $\mathbf{Mon_s}$ of unital semigroups and semigroup homomorphisms. By introducing suitable natural transformations, we turn these equivalences into 2-equivalences between 2-categories. The 2-category $\mathbf{Mon_s^{(2)}}$ constructed this way proves the good one to study Morita equivalence of monoids.
  \end{abstract}

\section*{Introduction}\label{Sec:Intro}

It is standard to consider a monoid $M$ as a one-object category (its `delooping' $\bB M$), where elements of the monoid become morphisms of the category. Or, even better, as a pointed and connected small category (where there is exactly one isomorphism class of objects). This viewpoint, in particular, leads to an equivalence between the ($2$-)category $\mathbf{Mon}$ of monoids and $\mathbf{ConnectedCat*}$, the ($2$-)category of pointed and connected small categories (see notably the lecture by Baez and Shulman on the subject \cite[Section 5.6]{Baez2007lectures}).

In this article, we take the opposite point of view and associate to any monoid a kind of `loop space', which is a small category whose objects are the elements of the monoid, described in Section \ref{Sec:1}. This construction is due to Costa and Steinberg \cite{costa2015Schutzenberger} in the semigroup case, but is actually a special case of a more general categorical construction when one considers monoids. This category, called the {S}ch{\"u}tzenberger category of the monoid, happens to be equipped canonically with a strict factorization system with additional properties. Such systems are investigated in Section \ref{Sec:SFS}. This crucial property allows us to construct a category equivalent to the category $\mathbf{Mon}$ of monoids and monoid homomorphisms in Section \ref{Sec:4}. This equivalence is then extended to the category $\mathbf{Mon_s}$ of unital semigroups and semigroup homomorphisms\footnote{There is a subtle distinction between monoids and unital semigroups. Monoids are (universal) algebras of type $(2,1)$, meaning they are defined by one binary operation (the associative product) and one constant element (the identity). Morphisms between monoids must preserve both the binary operation and the identity. On the other hand, unital semigroups are algebras of type $(2)$ (defined by one binary operation), with the extra property that there exists a two-sided identity. Morphisms between unital semigroups need not preserve this identity.}. By introducing suitable natural transformations, we finally extend these equivalences to 2-equivalences between 2-categories. It turns out that the 2-category $\mathbf{Mon_s^{(2)}}$ constructed this way has the following fundamental property: adjoint equivalence in this 2-category corresponds precisely to Morita equivalence of monoids. This is the content of Section \ref{Sec:Nat}.

\section{Small categories - Constructions}\label{Sec:0Cat}

We assume basic knowledge of category theory. If $\cC$ is a small category, we denote by $\cC_0$ its set of objects and by $\cC_1$ its set of arrows. For any two objects $a,b\in \cC$, we let $\cC(a,b)$ be the set of arrows with source $a$ and target $b$. In this article, composition of arrows will be denoted by juxtaposition: 
 \[\xymatrix{
a\ar@/^/[r]^f& b \ar@/^/[r]^g& c}=\xymatrix{
a\ar@/^/[rr]^{fg}& & c}\]

\subsection{Strict Factorization Systems}\label{Sec:SFS}

The following definition and results regarding strict factorization systems are taken from Grandis \cite{grandis2000weak}. 
Let $\cC$ be a small category. A \emph{strict factorization system} (SFS) on $\cC$ is a pair $(\cE,\cM)$ of wide subcategories (same class of objects) of $\cC$ such that any morphism $f\in \cC_1$ admits a unique decomposition $f=em$ with $e\in \cE_1$ and $m\in \cM_1$. 
In the following, by a \emph{category with a strict factorization system} (CatSFS) we will mean a triple $(\cC,\cE,\cM)$ where $(\cE,\cM)$ is a SFS on $\cC$.

For any SFS $(\cE,\cM)$ it holds that \cite{grandis2000weak}:
\benn
\item $\cE_1\cap \cM_1=\cC_0$ the set of identities;
\item $\cE_1$ and $\cM_1$ are orthogonal in $\cC$ (any morphism $e\in \cE_1$ has the left lifting property with respect to any $m\in \cM_1$);
\item there is a unique \emph{orthogonal factorization system} (OFS) $(\EE,\MM)$ spanned by $(\cE,\cM)$, where morphisms of $\EE$ (resp. $\MM$) are of the form $ei$ (resp. $im$), with $i$ an iso and $e\in \cE$ (resp. $m\in \cM$).
\een
Maps in $\cE$ will be depicted by $\lCE$ and maps in $\cM$  will be depicted by $\lCM$.
Duality for SFS  takes a simple form: $(\cE,\cM)$ is a SFS for $\cC$ iff $(\cM^{op},\cE^{op})$ is a SFS for $\cC^{op}$. 

The SFS $(\cE,\cM)$ is epi-monic (or proper) if all $\cE-$maps  are epi (in $\cC$) and all $\cM$-maps are monics (in $\cC$). 

\begin{example}\label{EX1}
Let $X$ be a set, and $\cP(X)_0$ its power set. Let $\cP(X)_1$ be the set of functions between subsets of $X$. 
Then any function $f:A\to B$ ($A,B\subseteq X$) can be uniquely factored into a surjection $e$ followed by an inclusion $m$.
\[\xymatrix{
\;A\;\ar@/^/[rr]^f \ar@/_/@<-.5ex>@{->>}[dr]_e & & \;B\;\\
&  \;Im(f)\;\; \ar@/_/@{>->}[ur]_m& }\]
The pair $(\cE=\{surjections\}, \cM=\{inclusions\})$ (from $X$ to $X$) is a strict factorization system on $\cP(X)$.
\end{example}

\begin{example}\label{EX2}
Let $\NN$ be the set of natural numbers. Define a category $\cC$ with objects $\cC_0=\NN$. For any $a,b\in \NN$, let $\cC(a,b)=\{x\in \NN |x \leq a \text{ and } x\leq b\}$. We denote each such arrow $x\in \cC(a,b)$ by $a\overset{x}{\longrightarrow} b$. Composition is given by   \[\xymatrix{
a\ar@/^/[r]^x& b \ar@/^/[r]^y& c}=\xymatrix{
a\ar@/^/[rr]^{x\wedge y}& & c}\]
Let $\cE$ be the wide subcategory of $\cC$ with $\cE(a,b)=\{b\}$ if $b\leq a$ and $\cE(a,b)=\emptyset$ otherwise, and $\cM$ be the wide subcategory of $\cC$ with $\cM(a,b)=\{a\}$ if $a\leq b$ and $\cM(a,b)=\emptyset$ otherwise. Then $(\cC,\cE,\cM)$ is a CatSFS. The unique factorization of $a\overset{x}{\longrightarrow} b$ is \[\xymatrix{
\;a\;\ar@/^/[rr]^{x} \ar@/_/@<-.5ex>@{->>}[dr]_{x} & & \;b\;\\
& \;x\;\; \ar@/_/@{>->}[ur]_{x}& }
\] 
\end{example}

\begin{example}\label{EX3}
Let $\NN$ be the set of natural numbers. Define a category $\cC$ with objects $\cC_0=\NN$. For any $a,b\in \NN$, let $\cC(a,b)=\{x\in \NN |a \leq x \text{ and } b\leq x\}$. We denote each such arrow $x\in \cC(a,b)$ by $a\overset{x}{\longrightarrow} b$. Composition is given by   \[\xymatrix{
a\ar@/^/[r]^x& b \ar@/^/[r]^y& c}=\xymatrix{
a\ar@/^/[rr]^{x-b+y}& & c}\]
Let $\cE$ be the wide subcategory of $\cC$ with $\cE(a,b)=\{b\}$ if $a\leq b$ and $\cE(a,b)=\emptyset$ otherwise, and $\cM$ be the wide subcategory of $\cC$ with $\cM(a,b)=\{a\}$ if $b\leq a$ and $\cM(a,b)=\emptyset$ otherwise. Then $(\cC,\cE,\cM)$ is a CatSFS. The unique factorization of $a\overset{x}{\longrightarrow} b$ is \[\xymatrix{
\;a\;\ar@/^/[rr]^{x} \ar@/_/@<-.5ex>@{->>}[dr]_{x} & & \;b\;\\
& \;x\;\; \ar@/_/@{>->}[ur]_{x}& }
\] 
(the diagram is commutative since $x-x+x=x$).
\end{example}

We now make some definitions. Let $(\cC,\cE,\cM)$ be a CatSFS.
\benn
\item The SFS $(\cE,\cM)$ is \emph{thin} if $\cE$ and $\cM$ are thin categories (there is at most one map between any two objects)\footnote{Thin categories are also called preorders.}. It is easy to prove that such thin SFSs are epi-monic;
\item The SFS $(\cE,\cM)$ is \emph{unital at $\zeta\in \cC_0$} if $\zeta$ is both a \emph{strict} initial point of $\cE$ and a \emph{strict} terminal point of $\cM$ (that is, for any $a\in \cC_0$, there exist \emph{strictly unique}\footnote{and not just 'unique up to an isomorphism', as it is customary in category theory} maps $\zeta \lCE a$ and $a\lCM \zeta$). This implies that $\cC$ is connected (no Hom-Set of $\cC$ is empty); 
\item The SFS is \emph{complete} if for any $e:a\lCE b\in \cE_1$ and $m:a\lCM c\in \cM_1$ there exists $d\in \cC_0$ and $e':c\lCE d\in \cE_1$, $m':b\lCM d\in \cM_1$, such that $em'=me'$, and dually for any $e':c\lCE d\in \cE_1$ and $m':b\lCM d\in \cM_1$ there exists $a\in \cC_0$ and $e':a\lCE b\in \cE_1$, $m':a\lCM c\in \cM_1$, such that $em'=me'$. Pictorially, we have
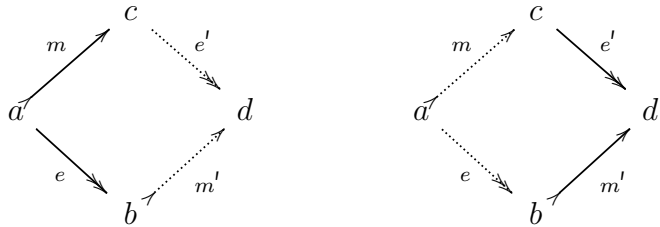
\begin{center}
\[\xymatrix{
& \;c\; \ar@{.>>}[dr]^{e'}& \\
\;a\; \ar@{>->}[ur]^m \ar@{->>}[dr]_e & & \;d\;\\
& \;b\; \ar@{>.>}[ur]_{m'} & }\qquad \qquad
\xymatrix{
& \;c\; \ar@{->>}[dr]^{e'}& \\
\;a\; \ar@{>.>}[ur]^m \ar@{.>>}[dr]_e & & \;d\;\\
& \;b\; \ar@{>->}[ur]_{m'} & }
\]
\end{center}
\een
If the SFS is complete and unital at $\zeta$ then it has the following two properties: 
\benn
\item for any $e:a\lCE x\in \cE_1$ there exists (a necessarily unique) $m':x\lCM u$ such that 
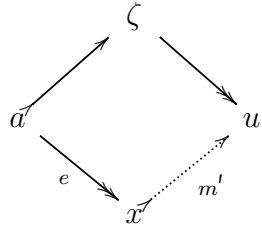
\[\xymatrix{
& \;\zeta\; \ar@{->>}[dr]& \\
\;a\; \ar@{>->}[ur] \ar@{->>}[dr]_e & & \;u\;\\
& \;x\; \ar@{>.>}[ur]_{m'} & }
\]  
\item Dually for any $m':x\lCM b\in \cM_1$ there exists (a necessarily unique) $e:v\lCE x\in \cE_1$ such that 
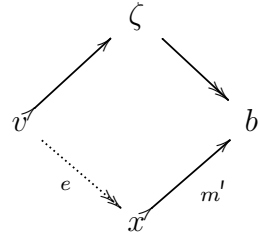
\[\xymatrix{
& \;\zeta\; \ar@{->>}[dr]& \\
\;v\; \ar@{>->}[ur] \ar@{.>>}[dr]_e & & \;b\;\\
& \;x\; \ar@{>->}[ur]_{m'} & }\]
\een
These two properties are in fact equivalent to completeness, as explained in the following diagram (the second case is dual):
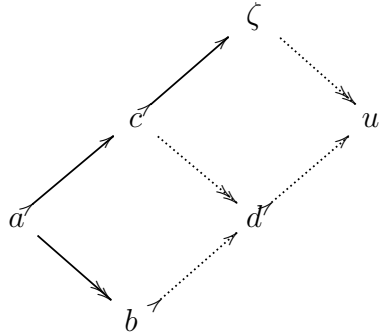
\begin{center}
\[\xymatrix{
&  &\;\zeta\;\ar@{.>>}[dr]& \\
& \;c\; \ar@{.>>}[dr] \ar@{>->}[ur] & &\;u\;\\
\;a\; \ar@{>->}[ur] \ar@{->>}[dr]& & d\ar@{>.>}[ur]&\\
& \;b\;\; \ar@{>.>}[ur] & &}
\]
\end{center}
Observe also that, in this case (where the SFS is unital), the completion is unique.

By a $CatSFS^*_t$ (resp.\  $CatSFS^*_c$, $CatSFS^*_{c,t}$) we mean a small category $(\cC,\zeta)$ with unital and thin (resp.\ unital and complete, resp.\ unital, complete and thin) factorization system $(\cE,\cM)$ and unit $\zeta$.

To better grasp the notions, let us consider the previous examples.

\begin{example}
We consider the setting of Example \ref{EX1}, with CatSFS \[(\cP(X), \cE=\{surjections\}, \cM=\{inclusions\}).\] 
\begin{itemize}
\item $\cM$ is a thin category, but $\cE$ is not thin unless $\sharp(X)\leq 1$ ($X$ is empty or a singleton). 
\item The object $X$ is strictly terminal in $\cM$, but not strictly initial in $\cE$ unless $X=\emptyset$.
\item The SFS is always complete. Indeed, let first be $e:A\lCE B\in \cE_1$ and $m:A\lCM C\in \cM_1$ ($e$ is a surjection from $A$ onto $B$ and $A$ is a subset of $C$). If $A$ is the empty set, then $B$ is also the empty set, and we may take $D=C$ and $C\lCE D$ the identity function. Otherwise, we let $D=B$, and define a surjection $e':C\lCE B$ as follows: for any $a\in A\subseteq C$, $e'(a)=e(a)$, and for $c\in C\backslash A$, $e(c)=e(a_0)$, where $a_0$ is any chosen element of $A$. Then the following diagram commute
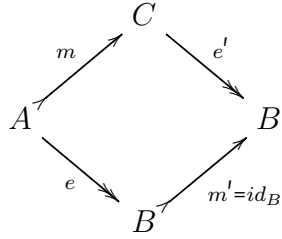
\[\xymatrix{
& \;C\; \ar@{->>}^{e'}[dr]& \\
\;A\; \ar@{>->}[ur]^m \ar@{->>}[dr]_{e}& & \;B\;\\
& \;B\; \ar@{>->}[ur]_{m'=id_B} & }\]
Second, let $e:C\lCE D\in \cE_1$ and $m:B\lCM D\in \cM_1$. Pose $A=e^{-1}(B)$ and define $e'$ as the restriction of $e$ from $A$ onto $B$. It holds that $A\subseteq C$, so that 
\[\xymatrix{
& \;C\; \ar@{->>}[dr]^e& \\
\;A\; \ar@{>->}[ur] \ar@{->>}[dr]_{e'}& & \;D\;\\
& \;B\; \ar@{>->}[ur]_m & }
\]
\end{itemize}
\end{example}

\begin{example}
We consider the CatSFS of Example \ref{EX2}: $\cC_0=\NN$, $\cC(a,b)=\{x\in \NN |x \leq a \text{ and } x\leq b\}$ and composition is defined by \[\xymatrix{
a\ar@/^/[r]^x& b \ar@/^/[r]^y& c}=\xymatrix{
a\ar@/^/[rr]^{x\wedge y}& & c}\]
The category $\cE$ satisfies that $\cE(a,b)=\{b\}$ if $b\leq a$ and $\cE(a,b)=\emptyset$ otherwise, and $\cM$ satisfies that $\cM(a,b)=\{a\}$ if $a\leq b$ and $\cM(a,b)=\emptyset$ otherwise. 
\begin{itemize}
\item As $\NN$ has no top element, we have that the SFS is not unital. 
\item The SFS is thin by definition of $\cE$ and $\cM$. 
\item It is also complete, as proved by the following commutative diagrams:  
\[b\leq a\leq c\qquad
\xymatrix{
& \;c\; \ar@{.>>}[dr]^b & \\
\;a\; \ar@/^{2.5cm}/@{.>}[rr]^{a\wedge b=b} \ar@/_{2.5cm}/@{.>}[rr]_{b\wedge b=b} \ar@{>->}[ur]^a \ar@{->>}[dr]_b& & \;b\;\\
& \;b\; \ar@{>.>}[ur]_b & }\]

\[b\leq d\leq c\qquad
\xymatrix{
& \;c\; \ar@{->>}[dr]^d& \\
\;b\;\ar@/^{2.5cm}/@{.>}[rr]^{b\wedge d=b} \ar@/_{2.5cm}/@{.>}[rr]_{b\wedge b=b} \ar@{>.>}[ur]^b \ar@{.>>}[dr]_b& & \;d\;\\
& \;b\; \ar@{>->}[ur]_b & }
\]
\end{itemize}
\end{example}

\begin{example}
We consider the CatSFS of Example \ref{EX3}: $\cC_0=\NN$, $\cC(a,b)=\{x\in \NN |a \leq x \text{ and } b\leq x\}$ and composition is defined by \[\xymatrix{
a\ar@/^/[r]^x& b \ar@/^/[r]^y& c}=\xymatrix{
a\ar@/^/[rr]^{x-b+ y}& & c}\]
The category $\cE$ satisfies that $\cE(a,b)=\{b\}$ if $a\leq b$ and $\cE(a,b)=\emptyset$ otherwise, and $\cM$ satisfies that $\cM(a,b)=\{a\}$ if $b\leq a$ and $\cM(a,b)=\emptyset$ otherwise. 
\begin{itemize}
\item As $\NN$ has a bottom element $0$, the SFS is unital at $\zeta=0$. 
\item The SFS is thin by definition of $\cE$ and $\cM$. 
\item It is also complete. To prove this, we use the characterization of completeness for unital SFS. Let $a\lCE x\in \cE_1$ ($a\leq x$). We let $u=x-a$ (observe that $u=x-a\in \NN$ since $a\leq x$, and that $x-a\leq x$ since $a\geq 0$). Then 
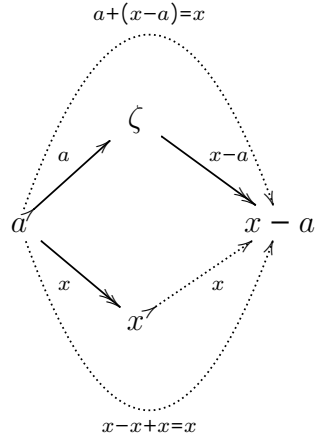
\[\xymatrix{
& \;\zeta\; \ar@{->>}[dr]^{x-a}& \\
\;a\;\ar@/^{2.5cm}/@{.>}[rr]^{a+(x-a)=x} \ar@/_{2.5cm}/@{.>}[rr]_{x-x+x=x} \ar@{>->}[ur]^a \ar@{->>}[dr]_{x} & & \;x-a\;\\
& \;x\; \ar@{>.>}[ur]_{x} & }
\] The other case is dual.
\end{itemize}
\end{example}

\subsection{Small categories associated to semigroups and monoids}\label{Sec:1}

In the sequel, $S=(S_0,.)$ will denote a semigroup (with $S_0$ its underlying set of elements), and $M=(M_0,.,1)$ will denote a monoid. By contrast, $(M_0,.)$ will be a unital semigroup. 
In order to study semigroups, various authors have introduced small categories associated to a semigroup $S$ in a canonical way. We consider here the \emph{{S}ch{\"u}tzenberger category} $\DD(S)$ of $S$ first defined in Costa and Steinberg \cite{costa2015Schutzenberger}.
Objects of $\DD(S)$ are elements of $S_0$ and morphisms are triples of the form $f=(a,x,b)$ with $x\in aS^1\cap S^1b$ (where $S^1$ denotes the unital semigroup generated by $S$). The domain of $f$ is $a$, its codomain is $b$ and we use the notation $f=a \overset{x}\longrightarrow b$. If $x=av=ub$ for some $u,v\in S^1$ and $g=b \overset{y}\longrightarrow c$ is a morphism with $y=bw=rc$, then the concatenation is $fg=a \overset{x}\longrightarrow b \overset{y}\longrightarrow c=a \overset{uy=xw}\longrightarrow c$\footnote{Our definition is actually opposite to that of \cite{costa2015Schutzenberger}}. In terms of \emph{Green's preorders} (see \cite{green1951structure, howie1976introduction}, and \cite{costa2015Schutzenberger, mary2021b} for the link with the {S}ch{\"u}tzenberger category): $a \overset{x}\longrightarrow b$ is a morphism in the category $\DD(S)$ if and only if $x\leq_{\cL} b$ and $x\leq_{\cR} a$.

For a monoid, this construction actually amounts to a certain completion of categories. Let $M=(M_0,.,1)$ be a monoid and $\bB M$ be its `delooping' (one-object category with elements of $M$ as morphisms, and product as composition). Let also $(\bB M) ^{\to}$ be its category of arrows, and define a relation $\rho$ on $(\bB M) ^{\to}_1$ by $(u,v)\rho (u',v')$, with $(u,v):a\to b$ and $(u',v'):a\to b$, if $av=av'$ and $ub=u'b$. Then the category $\DD(M)= (\bB M) ^{\to}/\rho$ is the \emph{Freyd completion}, or \emph{epi-monic completion}, of the category $\bB M$ \cite{grandis2000weak}. It is called the \emph{pre-regular completion} in \cite{hu1996note}. 

\textsc{Arrows in $\DD(M)$: } 
\begin{figure}[!ht]
\begin{center}
\[\xymatrix{
a\ar[d]_{x} & &a\ar[dd]^{xw=uy}\\
b \ar[d]_{y} &= &  \\
c& &\;c}\qquad \qquad
\xymatrix{
\bullet \ar[d]_{u} \ar[dr]^{x}  \ar[r]^{a} &   \bullet \ar[d]^{v}\\
\bullet \ar[d]_{r} \ar[dr]^{y}  \ar[r]^{b} &   \bullet \ar[d]^{w}\\
\bullet   \ar[r]_{c} &   \bullet}
\]
\end{center}
\caption{$\DD(M)$ and composition}\label{FigFrS}
\end{figure}

It turns out that the category $\DD(M)$ has a strict factorization system, that enjoys many properties:

\textsc{Strict factorization of $f:a\overset{x}{\rightarrow} b$ in $\DD(M)$:}  
\begin{figure}[!ht]
\begin{center}
\[\xymatrix{
a\ar[dd]_{x}\\
\;\\
b}\xymatrix{
\;\\
=\\
\; }\xymatrix{
a\ar@{->>}[d]^{x}\\
x\ar@{>->}[d]^{x}\\
b }\qquad \qquad
\xymatrix{
\bullet \ar[d]_{1} \ar[dr]^{x}  \ar[r]^{a} &   \bullet \ar[d]^{v}\\
\bullet \ar[d]_{u} \ar[dr]^{x}  \ar[r]^{x} &   \bullet \ar[d]^{1}\\
\bullet   \ar[r]_{b} &   \bullet}
\]
\end{center}
\caption{Strict factorization in $\DD(M)$}\label{FigFrSSFS}
\end{figure}

Maps of $\cE$ are of the form $a\overset{x}{\lCE} x$ (with $x=av$ for some $v\in M$) and maps of $\cM$ are of the form $x\overset{x}{\lCM} b$ (with $x=ub$ for some $u\in M$).
By definition of the canonical maps, the categories $\cE$ and $\cM$ are thin. The SFS is also unital at $\zeta=1$ ($1\overset{x}{\lCE} x$ exists and is unique for all $x\in M_0$, and dually), and complete. Completeness actually amounts to the commutation property of $\leqR$ and $\leqL$, well-known to semigroup theorists: if $a\leqR c\leqL b$ for some $c\in M_0$, then $a=cx$ and $c=yb$ for some $x,y\in M_0$, so that $a=ybx$ and $a\leqL bx\leqR b$ (and dually).

\begin{example}
Consider the monoid $(\NN, +, 0)$. Then the category $\DD(S)$ and its SFS are those defined in Example \ref{EX3}. 
\end{example}

\subsection{Monoids from unital SFS}

Let $A=(\cC,\cE,\cM,\zeta)$ be a category with a unital, strict factorization system. Let $X=\cC_0$ and define a binary operation $\ast$ on $X$ as follows. For any $(a,b)\in X^2$, $a\ast b$ is the central object in the (unique) factorization of the (unique) map $a\lCM \zeta\lCE b$.
 
\begin{proposition}
The binary operation $\ast$ is associative, and $(X,\ast,\zeta)$ is a monoid. Triple products $a\ast b\ast c$ are the central objects in the factorization of $a\lCM \zeta\lCE b\lCM \zeta \lCE c$.
\end{proposition}

\begin{proof}
Let $a,b,c\in X$. Then 
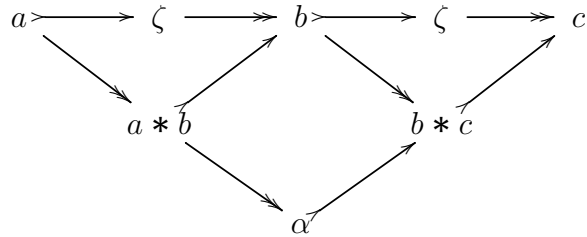
\[\xymatrix{
 \;a\;\ar@{->>}[dr] \ar@{>->}[r]& \;\zeta\;\ar@{->>}[r] &\;b\;\ar@{->>}[dr] \ar@{>->}[r]& \;\zeta\; \ar@{->>}[r]&\;c\;\\
& \;a\ast b\;\ar@{>->}[ur] \ar@{->>}[dr] &&\;b\ast c\;\ar@{>->}[ur]&\\
&  & \;\alpha\;\ar@{>->}[ur] &&}
\]
where the central element in the middle factorization is $\alpha=(a\ast b)\ast c=a\ast (b\ast c)$.\\
As $\zeta\lCM\zeta\lCE b=\zeta\lCE b\lCM b$ then $\zeta$ is a left identity. Dually it is a right identity and $(X,\ast,\zeta)$ is a monoid.
\end{proof}

The monoid $(X,\ast,\zeta)$ is actually isomorphic to $\cC_1(\zeta,\zeta)$ (with composition).
\begin{corollary}
\[(\cC_0,\ast,\zeta)\simeq \left(\cC_1(\zeta,\zeta),.\right).\]
\end{corollary}

\begin{proof}
We define an isomorphism as follows. Consider the function that sends $x\in \cC_0$ to $\zeta\lCE x\lCM \zeta$. It is bijective, and sends $\zeta$ to $id_{\zeta}$. Thus, we only have to prove that it is preserves the product (in which case its reciprocal will also preserve the product).
Let $x,y\in X$. Then 
\[\xymatrix{
\;\zeta\;\ar@{->>}[r] \ar@/_{.7cm}/@{->>}[drr] & \;x\;\ar@{->>}[dr] \ar@{>->}[r]& \;\zeta\;\ar@{->>}[r] &\;y\; \ar@{>->}[r]& \;\zeta\; \\
& & \;x\ast y\;\ar@{>->}[ur] \ar@/_{.7cm}/@{>->}[urr] &&}
\]
This ends the proof.
\end{proof}

We will use the notation $(X,\ast,\zeta)=\Sigma(A)$. For any monoid $M=(M_0,.,1)$, the following equality holds:
\[\Sigma\circ \DD(M)=M.\] Indeed, we have that \[a\overset{a=a.1}{\lCM} 1 \overset{b=1.b}{\lCE} b=a\overset{a.b}{\longrightarrow} b=a\overset{a.b}{\lCE} a.b\overset{a.b}{\lCM} b,\] so that $a\ast b=a.b$, for any $a,b\in M_0$.

\section{Some equivalences}\label{Sec:4}

\subsection{Categories and Functors}
We consider the following categories:
\benn
\item $\mathbf{Mon}$ is the large category of monoids and \textbf{monoid homomorphisms};
\item $\mathbf{Mon_s}$ is the large category of unital semigroups and \textbf{semigroup homomorphisms};  
\item $\mathbf{CatSFS^*_{c,t}}$ is the large category of small categories with a strict factorization system, such that the factorization system is unital, complete, and thin. Moreover we consider the small categories \textbf{pointed at their units}. An object of $\mathbf{CatSFS^*_{c,t}}$ is thus a quadruple $A=(\cC,\cE,\cM,\zeta)$, where $\cC$ is a small pointed category (at $\zeta$) with a SFS $(\cE,\cM)$ unital at $\zeta$. If $A$ and $A'$ are two objects of this category, then a morphism $H:A\to A'$ is a pointed functor $H:\cC\to \cC'$ that preserves the SFS ($H(\cE)\subseteq \cE'$ , $H(\cM)\subseteq \cM'$ and $H(\zeta)=\zeta'$);
\item By contrast $\mathbf{CatSFS_{u,c,t,s}}$ is the previous category but not pointed. The difference lies in the functors: if $A$ and $A'$ are two objects of this category, then a morphism $H:A\to A'$ is a functor $H:\cC\to \cC'$ that preserves the SFS ($H(\cE)\subseteq \cE'$ and $H(\cM)\subseteq \cM'$) and such that \[H(\zeta)\lCM\zeta'\lCE H(\zeta)=Id_{H(\zeta)}\] (equiv.\ $\zeta'\lCE H(\zeta)\lCM \zeta'$ is idempotent, equiv.\ $H(\zeta)\ast' H(\zeta)=H(\zeta)$). We call them \emph{semi-pointed functors}. 
\een

The assignment $\DD:M\mapsto \DD(M)$ extends to a functor on these categories in the obvious way: $\mathbf{\DD}:\mathbf{Mon}\to \mathbf{CatSFS^*_{c,t}}$ maps the monoid homomorphism $h:M\to M'$ to the functor $H=\mathbf{\DD}(h):\DD(M)\to \DD(M')$ with $H\left(a\overset{x}{\longrightarrow} b\right)=H(a)\overset{H(x)}{\lCE} H(b)$. In particular if $x=au$ for some $u\in X$, then  $H\left(a\overset{x}{\lCE} x\right)=H(a)\overset{H(x)}{\longrightarrow} H(x)$ with $H(x)=H(a)H(u)$.

Similarly, the assignment $\Sigma:A\mapsto \Sigma(A)$ extends to a functor $\mathbf{\Sigma}:\mathbf{CatSFS^*_{c,t}}\to \mathbf{Mon}$, that maps $H:A\to A'$ to its underlying function from $X=\cC_0$ to $X'=\cC^{'}_0$. As $H$ preserves the SFS and the units, it is a monoid homomorphism. Just apply $H$ to the commutative diagram \[\xymatrix{
 \;a\;\ar@{->>}[dr] \ar@{>->}[r]& \;\zeta\;\ar@{->>}[r] &\;b\;\\
& \;a\ast b\;\ar@{>->}[ur] &}
\] to obtain
\[\xymatrix{
 \;H(a)\;\ar@{->>}[dr] \ar@{>->}[r]& \;H(\zeta)=\zeta'\;\ar@{->>}[r] &\;H(b)\;\\
& \;H(a\ast b)=H(a)\ast'H(b)\;\ar@{>->}[ur] &}
\] 

As functors, we have the equality $\mathbf{\Sigma}\circ \mathbf{\DD}=\mathbf{Id_{Mon}}$. 

\subsection{Equivalences of categories}
We now prove that $\mathbf{Mon}$ is equivalent to $\mathbf{CatSFS^*_{c,t}}$.  

\begin{theorem}\label{Th_Equiv}
$\mathbf{\DD} \dashv  \mathbf{\Sigma} \colon \mathbf{Mon}\to \mathbf{CatSFS^*_{c,t}}$ is an adjoint equivalence.
\end{theorem}

\begin{proof}
The unit $\eta=\{\eta_M=id_M, M\in  \mathbf{Mon}\}$ is the identity from $\mathbf{Id_{Mon}}$ to $\mathbf{\Sigma} \circ \mathbf{\DD}=\mathbf{Id_{Mon}}$. Its components are all isomorphisms.\\
The counit $\varepsilon$ is defined for any $A=(\cC, \cE,\cM,\zeta)$ as follows: $\varepsilon_A:\mathbf{\DD\circ \Sigma}(A)\to A$ is the identity on object functor that sends any morphism $a\overset{x}{\lCE} x\overset{x}{\lCM} b$ to $a\lCE x\lCM b$ in $\cC_1$. As $\cE$ and $\cM$ are thin, these two maps $a\lCE x$ and $x\lCM b$ are unique if they exist. That these maps exist comes from the definition of the product $\ast$ and the construction of $\DD(M)$. As $a\overset{x}{\lCE} x$ then $x=a\ast v$ for some $v\in X=\cC_0$, hence $a\lCM \zeta\lCE v=a\lCE a\ast v\lCM v=a\lCE x\lCM v$. By dual arguments, $x\lCM b$ exists. One checks directly that it is a natural transformation, and that the triangle identities are trivially satisfied.\\
We finally prove that $\varepsilon_A$ is an isomorphism by constructing its inverse. It is the morphism $\varepsilon_A^{-1}:A\to \mathbf{\DD\circ \Sigma}(A)$ defined by the identity on object functor that sends $a\lCE x\lCM b$ to $a\overset{x}{\lCE} x\overset{x}{\lCM} b$. That $a\overset{x}{\lCE} x$ exists follows precisely from the definition of the completeness: as $a\lCE x$ then $a\lCM\zeta \lCE v=a\lCE x\lCM v$ for some $v\in \cC_0$, and $x=a\ast v$. And dually for $x\overset{x}{\lCM} b$.  
\end{proof}

Similarly we prove that $\mathbf{Mon_s}$ is equivalent to $\mathbf{CatSFS_{u,c,t,s}}$.  
First, we observe that $\mathbf{\DD}:\mathbf{Mon_s}\to \mathbf{CatSFS_{u,c,t,s}}$, that maps the semigroup homomorphism $h:M\to M'$ to the functor $H=\mathbf{\DD}(h):\DD(M)\to \DD(M')$ is still well defined.

Second, we can define a functor $\mathbf{\Sigma}:\mathbf{CatSFS_{u,c,t,s}}\to \mathbf{Mon_s}$, that maps the semi-pointed functor $H:A\to A'$ to its underlying function from $X=\cC_0$ to $X'=\cC'_0$. If we apply $H$ to the commutative diagram \[\xymatrix{
 \;a\;\ar@{->>}[dr] \ar@{>->}[r]& \;\zeta\;\ar@{->>}[r] &\;b\;\\
& \;a\ast b\;\ar@{>->}[ur] &}
\] we obtain
\[\xymatrix{
 \;H(a)\;\ar@{->>}[dr] \ar@{>->}[r]& \;H(\zeta)\;\ar@{->>}[r] &\;H(b)\;\\
& \;H(a\ast b)\;\ar@{>->}[ur] &}
\] 
But as $H(\zeta)\lCM\zeta'\lCE H(\zeta)=Id_{H(\zeta)}$, this can be rewritten in the form 
\[\xymatrix{
 \;H(a)\;\ar@{->>}[dr] \ar@{>->}[r]& \;\zeta'\;\ar@{->>}[r] &\;H(b)\;\\
& \;H(a\ast b)\;\ar@{>->}[ur] &}
\] 
so that $h(a\ast b)=h(a)\ast'h(b)$ ($h$ is a semigroup homomorphism).
As functors, we have the equality $\mathbf{\Sigma}\circ \mathbf{\DD}=\mathbf{Id_{Mon_s}}$.

\begin{theorem}\label{Th_Equivs}
$\mathbf{\DD} \dashv  \mathbf{\Sigma} \colon \mathbf{Mon_s}\to \mathbf{CatSFS_{u,c,t,s}}$ is an adjoint equivalence.
\end{theorem}

\begin{proof}
The proof is similar to that of Theorem \ref{Th_Equiv}.
\end{proof}

\section{Natural transformations, 2-categories and Morita equivalence}\label{Sec:Nat}

\subsection{Natural transformations and conjugations}
Consider two morphisms $F,G:A\to A'$ in the category $\mathbf{CatSFS^*_{c,t}}$ ( $A=(\cC,\cE,\cM,\zeta)$ and $A'=(\cC',\cE',\cM',\zeta')$). As pointed functors, we can consider pointed natural transformations between them.   
A pointed natural transformation $\alpha: F\Rightarrow G$ between pointed functors must satisfy:
\begin{center}
\[\xymatrix{
\;\zeta\; \ar@{->>}[d]  && \;\zeta'\;\ar@{->>}[r] \ar@{->>}[d]  & \;\zeta'\;\ar@{>->}[r] \ar@{->>}[d]& \;\zeta'\; \ar@{->>}[d]\\
\;x\;\ar@{>->}[d] &&\;F(x)\;\ar@{->>}[r]\ar@{>->}[d] & \;\alpha(x)\;\ar@{>->}[r] \ar@{>->}[d]& \;G(x)\; \ar@{>->}[d]\\
\;\zeta\; &&  \;\zeta'\;\ar@{->>}[r] &\zeta'\;\ar@{>->}[r]& \;\zeta'\;}
\] 
\end{center}
This implies that $F(x)=\alpha(x)=G(x)$ for all $x\in \cC_0$, and since $F$ and $G$ are entirely defined by their action on the objects, $F=G$, and $\alpha$ is the identity. 

\begin{corollary}
The $2$-category of monoids, monoid homomorphisms and trivial $2$-cells is equivalent to the $2$-category of small categories with unital, thin and complete SFS, pointed functors that respect the SFS as morphisms, and pointed natural transformations as $2$-cells.
\end{corollary}

As semi-pointed functors, however, the natural transformations will not degenerate. A natural transformation $\alpha: F\Rightarrow G$ amounts to a function $\alpha:\cC_0\to \cC_0^{'}$ such that the following diagram commutes:
\begin{center}
\[\xymatrix{
\;a\; \ar@{->>}[d]  && \;F(a)\;\ar@{->>}[r] \ar@{->>}[d]  & \;\alpha(a)\;\ar@{>->}[r] \ar@{->>}[d]& \;G(a)\; \ar@{->>}[d]\\
\;x\;\ar@{>->}[d] &&\;F(x)\;\ar@{->>}[r]\ar@{>->}[d] & \;\alpha(x)\;\ar@{>->}[r] \ar@{>->}[d]& \;G(x)\; \ar@{>->}[d]\\
\;b\; &&  \;F(b)\;\ar@{->>}[r] &\alpha(b)\;\ar@{>->}[r]& \;G(b)\;}
\] 
\end{center}
In particular, this function $\alpha$ defines a (non-pointed) functor that respects the SFS.

We let $\alpha=\alpha(\zeta)\in \cC_0^{'}$. Let $a=b=\zeta$ in this diagram. We obtain
\begin{center}
\[\xymatrix{
\;\zeta\; \ar@{->>}[d]  && \;F(\zeta)\;\ar@{->>}[r] \ar@{->>}[d]  & \;\alpha\;\ar@{>->}[r] \ar@{->>}[d]& \;G(\zeta)\; \ar@{->>}[d]\\
\;x\;\ar@{>->}[d] &&\;F(x)\;\ar@{->>}[r]\ar@{>->}[d] & \;\alpha(x)\;\ar@{>->}[r] \ar@{>->}[d]& \;G(x)\; \ar@{>->}[d]\\
\;\zeta\; &&  \;F(\zeta)\;\ar@{->>}[r] &\alpha\;\ar@{>->}[r]& \;G(\zeta)\;}
\] 
\end{center}

Since the functors $F$ and $G$ are semi-pointed, this reads (with the notation $\ast$) \[(\forall x\in \cC_0) \;\alpha(x)=F(x)\ast' \alpha=\alpha \ast' G(x).\] 
This expresses $\alpha$ as an `intertwiner' between $F$ and $G$.
In particular (for $x=\zeta$) it holds that $\alpha=F(\zeta)\ast' \alpha=\alpha\ast' G(\zeta)$. 

Conversely, given an element $\alpha\in \cC_0^{'}$ such that $(\forall x\in \cC_0) \;F(x)\ast' \alpha=\alpha \ast' G(x)$, we can define $\alpha(x)=F(x)\ast'\alpha=\alpha\ast' G(x)$ for all $x\in \cC_0$. As the SFS of $A'$ is thin and complete, all the maps involved in the diagram defining the natural transformation exist and commute. We consider such elements as the natural 2-cells of a (strict) 2-category $\mathbf{CatSFS_{u,c,t,s}^{(2)}}$, with  semi-pointed functors as 1-cells. 

Using the equivalence between $\mathbf{CatSFS_{u,c,t,s}}$ and $\mathbf{Mon_s}$, we can thus define natural transformations between semigroup homomorphisms. These natural transformations play a special role in the study of Morita equivalence of monoids, and have been called by M. Rogers conjugations (but his definition is opposite to ours).

\begin{definition}[opposite to {\cite[Definition 6.2]{rogers2019toposes}}] Let $f, g : M \to M'$ be semigroup homomorphisms. A \emph{conjugation} $\alpha$ from $f$ to $g$, denoted $\alpha : f \Rightarrow g$ is an element $\alpha\in M'$ such that
$ f(1)\alpha = \alpha = \alpha g(1)$ and for every $m\in M$, $f(m) \alpha  = \alpha g(m)$. 
\end{definition}

By construction, the functors $\mathbf{\DD}$ and $\mathbf{\Sigma}$ extend to 2-functors on the respective 2-categories, where the image of the conjugation $\alpha:f\Rightarrow g$ is the natural transformation $\alpha:F\to G$ defined by $\alpha(.)=F(.)\alpha=\alpha G(.)$, and conversely to $\alpha:F\to G$ we associate the conjugation $\alpha=\alpha(1)$.

\begin{corollary}
The 2-functors $\mathbf{\DD}$ and $\mathbf{\Sigma}$ define a strict 2-equivalence between the 2-categories $\mathbf{Mon_s^{(2)}}$, with conjugations as $2$-cells, and $\mathbf{CatSFS_{u,c,t,s}^{(2)}}$, with natural transformations as $2$-cells. 
\end{corollary}

\subsection{Morita equivalence of monoids}

By \cite[Corollary 6.6]{rogers2019toposes}, these 2-cells provides exactly the good notion of Morita equivalence between monoids (in the sense of \cite{banaschewski1972functors, knauer1971projectivity, talwar1995morita}). However, the triangle identities are not considered, and the proof uses topos theory. We will prove it directly below.

First, we recall the following well known characterization of Morita equivalence by enlargements, due independently to Knauer \cite[Theorem]{knauer1971projectivity} and Banachewski \cite[Proposition 4]{banaschewski1972functors}. Let $M,M'$ be monoids. Then $M$ and $M'$ are Morita equivalent if and only if there is an idempotent $e\in M$ such that $M = MeM$ and $M'\simeq eMe$.

To interpret this result in terms of adjoints in our $2$-category, we need the following lemma, that describes adjunctions in terms of conjugations.
\begin{lemma}
Let $f:M\to M', g:M'\to M$ be a pair of semigroup morphisms. Let also $\eta:Id_M\Rightarrow g\circ f$ and $\varepsilon: f\circ g \Rightarrow Id_{M'}$ be two conjugations. 
The triangle identities are satisfied (equiv.\ $\DD(f)$ and $\DD(g)$ are adjoint functors, equiv.\ $f$ and $g$ are adjoint in the 2-category $\mathbf{Mon_s^{(2)}}$) if and only if: 
\benn
\item $f(1)=f(\eta)\varepsilon$; 
\item $g(1')=\eta g(\varepsilon)$.
\een
\end{lemma}

\begin{proof}
The triangle identities are satisfied if and only if for all $m\in M$, $f(m)$ is the central element in the factorization of $f(\eta(m))\lCM f(g(f(m))) \lCE \varepsilon(f(m))$. Since $\eta:id_{M}\Rightarrow (g\circ f)$ is a conjugation, it satisfies $\eta(m)=m\eta=\eta g(f(m))$ for all $m\in M$, and dually $\varepsilon(f(m))=f(g(f(m)))\varepsilon=\varepsilon m$. Thus the central element in the factorization is $f(\eta) f(g(f(m)))\varepsilon=f(m)\f(\eta)\varepsilon$ and the first identity reduces to: for all $m\in M$, $f(m)=f(m)f(\eta)\varepsilon$. Dually, the second identity reads for all $m'\in M'$, $g(m')=\eta g(\varepsilon)g(m')$. But $f$ and $g$ are semigroup homomorphisms, so that this is equivalent to $f(1)=f(\eta)\varepsilon$ and $g(1')=\eta g(\varepsilon)$. 
\end{proof}

We also need to characterize invertible conjugations.
\begin{lemma}\label{LemIConj}
Let $f,g:M\to M'$ be two homomorphisms, and $\alpha:f\Rightarrow g$ be a conjugation.Then $\alpha$ is invertible if and only if there exists $\beta\in M'$ such that $\alpha\beta=f(1)$ and $\beta\alpha=g(1)$. In this case, the inverse of $\alpha:f\Rightarrow g$ is the conjugation $\gamma:g\Rightarrow f$ defined by the element $\gamma=\beta\alpha\beta$.
\end{lemma}
(Contrary to \cite{rogers2019toposes}, we do not assume \emph{a priori} that $\beta$ is a conjugation).
\begin{proof}
Consider $F,G:\DD(M)\to \DD(M')$. Then $\alpha: F\Rightarrow G$ is invertible if and only if there exists a conjugation $\beta$ such that their composite is the identity transformation. 
By composition, this implies that for all $x\in M$ $f(x)=\alpha g(x)\beta=f(x)\alpha\beta$ and in particular $f(1)=f(1)\alpha\beta$. But $f(1)\alpha=\alpha$ and we conclude that $f(1)=\alpha\beta$. The second identity is dual. \\
Conversely, suppose that there exists an element $\beta$ satisfying these identities. We prove that $\gamma=\beta\alpha\beta$ is a conjugation from $g\to f$. 
Let $m\in M$. Then \begin{eqnarray*}
g(m)\beta &=& g(1.m)\beta\\
&=& g(1)g(m)\beta\\
&=& \beta\alpha g(m)\beta\\
&=& \beta f(m)\alpha \beta \qquad \text{ (since $\alpha:f\to g$ is a conjugation)}\\
&=&\beta f(m)f(1)\\
&=&\beta f(m.1)\\
&=&\beta f(m)
\end{eqnarray*}
Thus $g(m)\gamma=g(m)\beta\alpha\beta=\beta g(m)g(1)=\beta g(1)g(m)=\gamma g(m)$. As also $g(1)\gamma=\gamma=\gamma f(1)$, we deduce that $\gamma$ is a conjugation from $g$ to $f$. 
Finally, we have that $\alpha\gamma=\alpha\beta\alpha\beta=f(1)\alpha\beta=\alpha\beta=f(1)$, and dually $\gamma\alpha=g(1)$. We deduce that the associated natural transformations are reciprocal.  \end{proof}

Observe that, under the assumptions of Lemma \ref{LemIConj}, $\beta$ is an inner inverse of $\alpha$ in the monoid $M'$. In particular, $\alpha\beta, \beta\alpha$ are idempotents and the inverse conjugation $\gamma=\beta\alpha\beta$ is a reflexive inverse of $\alpha$. Also, we see that in order to be related by an invertible conjugation, two homomorphisms must necessarily satisfy that $f(1)$ and $g(1)$ are isomorphic idempotents (where two idempotents $e,f$ in a monoid $M'$ are \emph{isomorphic} if there exists $x,y\in M'$ such that $e=xy$ and $f=yx$, equiv.\ $eM'\simeq fM'$ as right $M'$-acts, equiv.\ $e\cD f$ where $\cD$ denotes Green's equivalence relation).

\begin{example}
Let $S(2)$ be the symmetric group on $\{1,2\}$ and $T(4)$ be the transformation semigroup on $\{1,2,3,4\}$. Let $f:S(2)\to T(4)$ than sends $(a\;b)$ to $(a\; b\; 3\;4)$, and $g:S(2)\to (T(3)$ than sends $(a\;b)$ to $(a\; b\; 3\;3)$. $f$ and $g$ are semigroup homomorphisms. Consider $\alpha=(1\;2\;3\;3)$. We claim that $\alpha:f\Rightarrow g$ is a conjugation. Since $S(2)$ contains only two elements, we check directly the identity $f(m) \alpha  = \alpha g(m)$ on each element $m=(1\;2)$ and $m=(2\;1)$.
\begin{eqnarray*}
f(1\;2)\alpha&=&(1\; 2\; 3\;4)(1\; 2\; 3\;3)=(1\; 2\; 3\;3)=(1\; 2\; 3\;3)(1\; 2\; 3\;3)=\alpha g(1\;2)\\
f(2\;1)\alpha&=&(2\; 1\; 3\;4)(1\; 2\; 3\;3)=(2\; 1\; 3\;3)=(1\; 2\; 3\;3)(2\; 1\; 3\;3)=\alpha g(2\;1)
\end{eqnarray*}
The conjugation $\alpha$ cannot be invertible because $f(1\;2)=(1\;2\;3\;4)$ and $g(1\;2)=(1\;2\;3\;3)$, and these transformations are not isomorphic idempotents (and thus $\alpha\beta=f(1\;2)$ and $\beta\alpha=g(1\;2)$ is not solvable).\\
If we consider the third homomorphism $h:S(2)\to T(4)$ than sends $(a\;b)$ to $(a\; b\; 4\;4)$, then we see that $\alpha:h\to g$ defined by the element $\alpha=(1\; 2\; 3\;3)$ is an invertible conjugation, with inverse $\beta:g\to h$ defined by the element $\beta=(1\;2\;4\;4)$. 
\end{example}

We are finally in position to consider Morita equivalence.
Let $M$ be a monoid, $e\in E(M)$ be an idempotent such that $MeM=M$, and  $M'\simeq eMe$. Without loss of generality, we can assume that $M'=eMe$. Let $x,y\in M$ be such that $xey=1$, and put $\eta=xe, \beta=ey$. Define $f:M\to eMe$ by $f(m)=eymxe$ for all $m\in M$, and $g:M'=eMe\to M$ by $g(m')=m'$. $f$ and $g$ are semigroup homomorphisms. It holds that $g\circ f:M\to M$, $m\mapsto eymxe$ and $f\circ g:eMe\to eMe$, $m'\mapsto eym'xe$. $\eta:Id_M\to g\circ f$ is a conjugation since $\eta=\eta (g\circ f)(1)$ ($xe=xee$) and for all $m\in M$, $m\eta=\eta (g\circ f)(m)$ ($mxe=xe(eymxe)$). Dually $\beta:g\circ f\to Id_{M}$ is a conjugation, and the reciprocal of $\eta$. Similarly, let $\varepsilon=\beta e=eye$. Then $\varepsilon:f\circ g\to Id_{M'}$ is a conjugation since $(f\circ g)(1')\varepsilon=\varepsilon$ ($(eyexe)eye=eye$) and for all $m\in M$, $(f\circ g)(m')\varepsilon=\varepsilon m'$ ($(eym'xe)eye=eym'e=eym'$). It is also invertible with inverse $e\eta$. Finally, we consider the triangle identities. 
We have that $f(1)=eyxe$  and $f(\eta)\varepsilon=(eyxexe)(eye)=eyxe$, so that $f(1)=f(\eta)\varepsilon$. Dually $g(e)=e=(xe)(eye)=\eta g(\varepsilon)$.
This proves that $M$ and $M'=eMe$ are equivalent in the 2-category $\mathbf{Mon_s^{(2)}}$.

Conversely, suppose that $M$ and $M'$ are equivalent in the 2-category. Then there exist two semigroup homomorphisms $f:M\to M'$ and $g:M'\to M$, and two invertible conjugations $\eta\in M$ and $\varepsilon\in M'$, such that $\eta=\eta g(f(1))$ and for all $m\in M$, $m\eta=\eta g(f(m))$, and $f(g(1'))\varepsilon=\varepsilon$ and for all $m'\in M'$, $f(g(m'))\varepsilon =\varepsilon m'$.
The triangle identities read $f(1)=f(\eta)\varepsilon$ and $g(1')=\eta g(\varepsilon)$. Put $e=g(1')$ (so that $\eta g(\varepsilon)=e$).
Let $\beta$ be the inverse of $\eta$. Then $\eta\beta=1$ and $\beta\eta=(g\circ f)(1)$. 
Let also $\mu$ be the inverse if $\varepsilon$. Then $\varepsilon\mu=(f\circ g)(1')=f(e)$ and $\mu\varepsilon=1'$. 
Moreover, as it is classical for units and counits of adjoint equivalences, their reciprocal are respectively counits and units of the adjoint equivalence in reverse order, so that $e=g(1')=g(\mu)\beta$ and $f(1)=\mu f(\beta)$. 
The element $e=g(1')$ is an idempotent of $M$, unit of the monoid $eMe$, and $g(M')\subseteq eMe$. Indeed, let $m'\in M'$. Then $g(m')=g(1'm'1')=eg(m')e$. For its reciprocal, define $h:eMe\to M'$ by $h(m)=\mu f(m)\varepsilon$, and let $m'\in M'$. Then $h(g(m'))=\mu f(g(m'))\varepsilon=\mu\varepsilon m'=m'$ (since $\varepsilon:f \circ g\to Id_{M'}$ is a conjugation). Let now $m\in eMe$. Then \[g(h(m))=g(\mu) g(f(m))g(\varepsilon)=g(\mu)g(f(m))g(f(1))g(\varepsilon)\] 
and using that $\beta\eta=(g\circ f)(1)$ we obtain
\[g(h(m))=g(\mu)g(f(m))\beta\eta g(\varepsilon)=g(\mu)g(f(m))\beta e=g(\mu)\beta m e= eme.\]
This proves that $M'\simeq eMe$.
Our final task is to prove that $MeM=M$. But $1=\eta\beta$ and $\eta=\eta g(f(1))=\eta g(1'f(1))=\eta g(1') g(f(1))$, so that $1=\eta e g(f(1)) \beta$.

We have proved:
\begin{theorem}
Two monoids $M$ and $M'$ are Morita equivalent if and only if there exists an adjoint equivalence $f:M\to M':g$ in the 2-category $\mathbf{Mon_s^{(2)}}$.
\end{theorem}

\section{Acknowledgements}
The content of this article was presented in a plenary talk at ICSOA 2024 (Kerala University, India). I once again express my gratitude to the organizers, especially A.R. Rajan and P.G. Romeo, for their kind invitation.\\
I also wish to express my heartfelt thanks to my colleague F. Metayer, for numerous helpful conversations on category theory. 
Finally, I would like to thank the anonymous referee for a careful reading of the manuscript and for insightful remarks, which significantly improved the clarity and presentation of this paper.

\bibliographystyle{plainnat}
\bibliography{BiblioRingsSemigroupsCat22}

\end{document}